\documentclass{article}
\usepackage{amssymb}
\usepackage{amsmath}
\usepackage{amsthm}

\theoremstyle{definition}
\newtheorem{definition}{Definition}[section]
\newtheorem{example}[definition]{Example}

\newtheorem{remark}[definition]{Remark}
\theoremstyle{plain}
\newtheorem{lemma}[definition]{Lemma}
\newtheorem{proposition}[definition]{Proposition}
\newtheorem{theorem}[definition]{Theorem}
\newtheorem{corollary}[definition]{Corollary}

\DeclareMathOperator{\Sing}{Sing}

\begin{document}
\title{\textbf{Singular hypersurfaces possessing infinitely many star points}}
\author{Filip Cools\footnote{K.U.Leuven, Department of Mathematics,
Celestijnenlaan 200B, B-3001 Leuven, Belgium, email:
Filip.Cools@wis.kuleuven.be} \hspace{0.2mm} and \addtocounter{footnote}{5}Marc Coppens\footnote{Katholieke
Hogeschool Kempen, Departement Industrieel Ingenieur en Biotechniek,
Kleinhoefstraat 4, B-2440 Geel, Belgium; K.U.Leuven, Department of Mathematics,
Celestijnenlaan 200B, B-3001 Leuven, Belgium; email: Marc.Coppens@khk.be}}

\maketitle {\footnotesize \emph{\textbf{Abstract.---} We prove that a component $\Lambda$ of the closure of the set of star points on a hypersurface of degree $d\geq 3$ in $\mathbb{P}^N$ is linear. Afterwards, we focus on the case where $\Lambda$ is of maximal dimension and the case where $X$ is a surface.\\ \\
\indent \textbf{MSC.---} 14J70, 14N15, 14N20}}
\\ ${}$

\section{Introduction}

Consider a hypersurface $X$ of degree $d\geq 3$ in a projective space $\mathbb{P}^N$ defined over the field $\mathbb{C}$ of complex numbers. A smooth point on $X$ is called a star point if and only if the intersection of $X$ with the embedded tangent space $T_P(X)$ is a cone with vertex $P$. As explained in \cite{CoCo}, this notion is a generalisation of total inflection points on plane curves. It is also a generalisation of the classical notion of an Eckardt point on a smooth cubic surface in $\mathbb{P}^3$ (see \cite{Eck,Ngu1,Ngu2}). Star points on smooth hypersurfaces have been studied in \cite{CoCo}, where it is proven that such hypersurfaces contain only finitely many star points. Star points on singular cubic surfaces in $\mathbb{P}^3$ have been examined in \cite{Cia,Ngu2,Ngu3}. In particular, in \cite{Cia} the author investigates a singular cubic surface containing infinitely many star points and makes some general thoughts on it.

In Section \ref{sec main thm} of this article, we prove that each component $\Lambda$ of the closure of the set of star points of $X$ is a linear subspace of dimension $0\leq \lambda\leq N-2$. Moreover, each smooth point $P$ in $\Lambda$ has the same tangent space $\Pi$ and $\Pi\cap X$ is a cone with vertex $\Lambda$. Note that for smooth hypersurfaces, we always have $\lambda=0$. In Section \ref{sec lambda maximal}, we study the case where $\Lambda$ is of maximal dimension, i.e. $\lambda=N-2$. We prove that either all such $(N-2)$-dimensional linear spaces $\Lambda$ belong to the same hyperplane in $\mathbb{P}^N$ or they contain a common $(N-3)$-dimensional linear space. In the first case, there are at most $d$ such linear spaces $\Lambda$; in the second case, there are at most $3d$ such linear spaces $\Lambda$. To finish the article, in Section \ref{sec surface} we consider surfaces in $\mathbb{P}^3$ having a line $L$ of star points. We show that in general such line contains $d-1$ singularities of $X$ of type $A_{d-1}$. Conversely, the presence of such singularities implies that $L$ is a line of star points. We also make some remarks in more special cases. On the one hand, our results give contradictions to the thoughts in \cite{Cia}; on the other hand, we place them in a much more general situation. Finally, we show that star points on a line of star points can be considered as a limiting case of a special type of isolated star points. This generalizes results of \cite{Ngu4} for cubic surfaces.

\section{Main theorem} \label{sec main thm}

We work over the field $\mathbb{C}$ of complex numbers.

\begin{definition}
Let $X$ be an irreducible reduced hypersurface of degree $d\geq 3$
in $\mathbb{P}^N$ and let $P$ be a smooth point on $X$. We say that $P$ is a {\it star point} on $X$ if and only if $T_P(X)\cap X$ is a cone with vertex $P$.
\end{definition}

\begin{example} \label{ex general} Assume that $X$ is an irreducible reduced hypersurface of
degree $d\geq 3$ in $\mathbb{P}^N$ and $\Pi$ is an hyperplane in
$\mathbb{P}^N$ such that the scheme $\Pi\cap X$ is a cone with
vertex $\Lambda$ (here $\Lambda$ is a linear subspace of $\Pi$). In
case $\Lambda$ is not contained in $\Sing (X)$, then all points of
$(X\setminus \Sing (X))\cap \Lambda$ are star points of $X$. In
particular, $\Lambda$ is contained in the closure of the set of star
points on $X$. In case $P\in \Lambda$ is a smooth point of $X$, then
$T_P(X)=\Pi$, hence all those star points have the same tangent space
to $X$.
\end{example}

\begin{example} \label{ex cone} Let $V$ be a plane in $\mathbb{P}^N$ and let $\Gamma$
be an irreducible plane curve of degree $d$ in $V$. Consider an $(N-3)$-dimensional linear subspace $L$ of $\mathbb{P}^N$ such that
$L \cap V =\emptyset$ and let $X$ be the cone on $\Gamma$ with
vertex $L$. If $Q$ is a total inflection point of $\Gamma$
with tangent line $T$, then each point $P\in \langle Q, L
\rangle \setminus L$ is a star point of $X$ and
$T_P(X)=\langle T, L \rangle$. In particular, $\langle Q,
L \rangle$ is contained in the closure of the locus of star
points on $X$ and all those star points have the same tangent space.
\end{example}

\begin{remark}
It should be noted that Example \ref{ex cone} is a special case of Example \ref{ex general}.
Here $\Pi=\langle T, L\rangle$ and $\Pi\cap X$ is as a divisor on $\Pi$ equal to
$d\Lambda$ with $\Lambda=\langle Q, L \rangle$, hence it is a
cone with vertex $\Lambda$.
\end{remark}

The following theorem implies that Example \ref{ex general} is the typical example.

\begin{theorem} \label{thm main}
Let $X$ be an irreducible reduced hypersurface of degree
$d$ in $\mathbb{P}^N$ and let $\Lambda$ be a component of the closure of
the locus of star points on $X$. Then $\Lambda$ is a linear subspace of
$\mathbb{P}^N$ of some dimension $0\leq \lambda\leq N-2$ and there is a
hyperplane $\Pi$ in $\mathbb{P}^N$ containing $\Lambda$ such that $\Pi\cap X$
is a cone in $\Pi$ with vertex $\Lambda$. In particular, $\Pi$ is the tangent space to $X$ at all smooth points of $X$ contained in $\Lambda$. Moreover, in case $\dim(\Lambda)=N-2$ then $\Pi\cap X$ is equal to $d\Lambda$ as a divisor on $\Pi$.
\end{theorem}

\begin{corollary}
The set of hyperplanes that do occur as tangent spaces
to $X$ at star points is finite.
\end{corollary}

\begin{proof}[Proof of Theorem \ref{thm main}]
Assume $\Lambda$ is a component of the space of star points of $X$ of dimension at least one. Let $P,P'$ be general points of $\Lambda$. As a
scheme, $T_P(X)\cap X$ is a cone with vertex $P$. In particular, if
$H$ is a hyperplane in $T_P(X)$ with $P\notin H$, then $H\cap X$ as a
scheme is a (not necessarily reduced or irreducible) hypersurface of
degree $d$ in $H$ such that $X\cap T_P(X)$ as a scheme is the
cone on $H\cap X$ with vertex $P$. Let $$\pi :(T_P(X)\cap X)\setminus
\{P \}\rightarrow H\cap X$$ be the projection map. Let $C_1, \ldots,
C_k$ be the irreducible components of $T_P(X)\cap X$. Moving $P$ to
$P'$ on $\Lambda$, we obtain corresponding
components $C'_1, \ldots, C'_k$ of $T_{P'}(X)\cap X$. We say a
component $C_i$ is moving (resp. not moving) with $P$ on $\Lambda$ in
case $C'_i\neq C_i$ (resp. $C'_i=C_i$). If $C_i$ is a moving component, the closure of the union of the components $C_i$ is equal to $X$.

Assume that $T_{P'}(X)=T_P(X)$. In particular, we have that $\Lambda \subset T_P(X)$. Since $P'$ is a star point of
$X$, it follows that for each point $Q$ on $T_P(X)\cap X$ one has $\langle
P',Q\rangle \subset X$, hence $\langle \pi (P'),\pi (Q)\rangle
\subset X\cap H$. This implies $X\cap H$ is a cone with vertex $\pi
(P')$ and therefore $X\cap T_P(X)$ is a cone with vertex $\langle P,
P'\rangle$. This holds for any $P'\in \Lambda$, hence $T_P(X)\cap X$
is a cone with vertex $\langle \Lambda \rangle$. It follows that
$\langle \Lambda \rangle \subset X$ and in case $P''\in \langle \Lambda
\rangle$ is a smooth point on $X$, then $P''$ is a star point on $X$
and $T_{P''}(X)=T_P(X)$. This is in agreement with the statement of the
theorem.

From now on, we assume $T_P(X)\neq T_{P'}(X)$. If moreover all components of $T_P(X)\cap X$ would be cones with vertex $T_P(\Lambda)$, then a general point of $T_P(\Lambda)$ is a star point of $X$, therefore $\Lambda=T_P(\Lambda)$. Thus $\Lambda$ is a linear subspace and $T_P(X)\cap X$ is singular at each point of $\Lambda$, hence $T_{P'}(X)=T_P(X)$, a contradiction. This implies that $T_P(X)\cap X$ has a component $C_P$ moving with $P\in \Lambda$ such that $C_P$ is not a cone with vertex $T_P(\Lambda)$. Indeed, if $T_P(X)\cap X$ has a component $C$ that does not move with $P$ on $\Lambda$, then $C=T_P(X)\cap T_{P'}(X)$. It follows that $C$ is a linear subspace of dimension $N-2$ contained in $X$ and containing $\Lambda$, hence also a cone with vertex $T_P(\Lambda)$.

Let $Q$ be a general point on $C_P$. The proof of \cite[Theorem 4.2]{CoCo} implies
that the line $\langle P, Q\rangle$ contains a singular point $R$
of $X$. As a matter of fact, in the case of smooth hypersurfaces this holds for any component of $T_P(X)\cap X$ using the fact that no such component is a cone with vertex $T_P(\Lambda)$. So in the case of singular hypersurfaces, we needed to show the existence of such a component in advance. It follows that the scheme $T_P(X)\cap X$ is singular at
$R$, hence $H\cap X$ is singular at $\pi(R)=\pi(Q)$. Since $\pi(Q)$ is a general point of the component $\pi(C_P)$ of $H\cap X$, it
follows that $\pi(C_P)$ is a multiple component of $H\cap X$ and
therefore $C_P$ is a multiple component of $T_P(X)\cap X$. In
particular, the scheme $T_P(X)\cap X$ is singular at $Q$. Since $X$
is smooth at $Q$, it follows that $T_P(X)=T_Q(X)$. Since $C_P$ is moving with $P\in\Lambda$, this implies that the Gauss map of $X$ has fibers of dimension $N-2$. Since the general fibers of the Gauss map are linear subspaces (see \cite[Theorem 2.3(c)]{Zak}), it follows that $C_P$ is linear subspace of dimension $N-2$ in $\mathbb{P}^N$. In case $T_P(X)\cap X$ would also have a component $C$ not moving with $P\in\Lambda$, it follows that a general hyperplane $\Pi$ of $\mathbb{P}^N$ containing the linear subspace $C$ is the tangent hyperplane of a star
point $P'$ on $\Lambda$. Since $T_{P'}(X)\cap X$ contains also the moving multiple component $C_{P'}\neq C$, this contradicts Bertini's Theorem for linear systems on singular varieties (see \cite[Theorem 4.1]{Kle}). So from now on, we can also assume that all components of $T_P(X)\cap X$ are moving, one of which is the linear subspace $C_P$.

Assume that $P'\not\in T_P(X)$ for $P'\in\Lambda$ general, so $H=T_P(X)\cap T_{P'}(X)$ is a hyperplane of $T_P(X)$ not passing through $P$. In particular, we have that $H\neq C_P$, thus the linear subspace $H'=C_P\cap H$ is $(N-3)$-dimensional. Let $L$ be a general line in $C_P$ through $P$ and let $Q\in H'$ be the intersection point of $H$ and $L$. Since $P'$ is a star point and $Q\in T_{P'}(X)$, we have $\langle Q,P'\rangle\subset X$, hence $P'\in T_Q(X)$. On the other hand, we have $T_P(X)\subset T_Q(X)$, hence $T_Q(X)=\mathbb{P}^N$ and $Q$ is a singular point of $X$. In particular, the point $Q$ does not move with $P'$ on $\Lambda$, since $L$ contains only finitely many singular points of $X$. We conclude that $H'$ does not move with $P'\in \Lambda$, hence $H'\subset T_{P''}(X)$ for
$P''$ general on $\Lambda$. Since $P''$ is a star point on $X$, we find
$\langle H', P''\rangle \subset X$ for $P''\in \Lambda$. In case
$\langle H', P''\rangle$ would not move with $P''$ on $\Lambda$, it would
imply $\Lambda \subset T_P(X)$, hence a contradiction. Therefore
$\langle H',P''\rangle$ moves with $P''$ on $\Lambda$ and this implies
$X$ is the closure of the union of those linear spaces of dimension $N-2$ inside $\mathbb{P}^N$.
This proves $X$ is a cone with vertex $H'$ on a plane curve $\gamma$ of degree
$d\geq 3$. Now star points correspond to points in $\langle
H',P''\rangle$ with $P''$ a total inflection point of $\gamma$. This
would imply $\Lambda \subset \langle H',P''\rangle$ for such a point
$P''$ hence $\Lambda \subset T_{P''}(X)$, giving a contradiction.

So we find that $P'\in T_P(X)$, hence $\langle\Lambda\rangle\subset T_P(X)\cap X$.
Let $D_P$ be a component of $T_P(X)\cap X$ containing $\langle\Lambda\rangle$. If $D_P$ is not a cone with vertex $T_P(\Lambda)$, then $D_P$ is a multiple component, hence $\Lambda\subset \text{Sing}(T_P(X)\cap X)$. We obtain $T_{P'}(X)=T_P(X)$, a contradiction. So we have that each component $D_P$ of $T_P(X)\cap X$ containing $\langle\Lambda\rangle$ is a cone with vertex $T_P(\Lambda)$. In particular, there exists a component $D_P\neq C_P$.

Assume that $T_P(\Lambda)\subset \text{Sing}(T_P(X)\cap X)$. If $P'\in C_P$ for $P'$ general, it follows that $T_{P'}(X)=T_P(X)$ since $C_P$ is a multiple component of $T_P(X)\cap X$, a contradiction. Hence we have that $P'\not\in C_P$ and also $P\not\in C_{P'}$. Since $T_P(\Lambda)\subset \langle\Lambda\rangle\subset T_{P'}(X)$ and $C_{P'}$ is a hyperplane in $T_{P'}(X)$ not through $P$, we find that each line $L$ in $T_P(\Lambda)$ containing $P$ intersects $C_{P'}$ in a point $Q_{P'}\neq P$. If $Q_{P'}$ is a smooth point on $X$, then $T_{Q_{P'}}(X)=T_{P'}(X)$ since $C_{P'}$ is a multiple component. On the other hand, $T_{Q_{P'}}(X)=T_P(X)$ since $Q_{P'}\in L\subset \text{Sing}(T_P(X)\cap X)$. This would imply that $T_{P'}(X)=T_P(X)$, a contradiction. Therefore, the point $Q_{P'}$ is singular on $X$ and because $L\cap\text{Sing}(X)$ is finite, it follows that $Q_{P'}$ does not move with $P'\in\Lambda$, hence $Q_{P'}\in C_P$. But $C_P$ is an $(N-2)$-dimensional linear subspace containing $P$, hence $L\subset C_P$ and thus $T_P(\Lambda)\subset C_P$, a contradiction.

Finally, assume that $T_P(\Lambda)\not\subset \text{Sing}(T_P(X)\cap X)$. Hence there is only one component $D_P$ of $T_P(X)\cap X$ having vertex $T_P(\Lambda)$. Moreover, the component $D_P$ is simple and it is an $(N-2)$-dimensional linear subspace (otherwise $T_P(\Lambda)$ would be contained in the singular locus). Take $Q\in\Lambda$ general. Since $Q\in\langle\Lambda\rangle\subset D_P$ and $D_P$ is a linear space on $X$, it follows that  $D_P\subset T_Q(X)$. But the union of the spaces $D_P$ is $X$, thus $X\subset T_Q(X)$, a contradiction.
\end{proof}

\begin{example} \label{example coordinates}
Let $X$ be a hypersurface in $\mathbb{P}^N$ and assume a linear subspace $\Lambda$
of dimension $0\leq \lambda\leq N-2$ is a component of the closure of the set of star points on $X$. Let $\Pi$ be the tangent space $T_P(X)$ of a smooth point $P\in \Lambda$. We can choose projective coordinates $(X_0:X_1:\ldots:X_N)$ on $\mathbb{P}^N$ such that $\Lambda$ has equation $X_{\lambda+1}=\ldots=X_N=0$ and $\Pi$ as equation $X_{\lambda+1}=0$. We can write the defining polynomial $f$ of $X$ as $X_{\lambda+1}.h+g$, where $g$ is independent of the variable $X_{\lambda+1}$. The intersection of $X$ with the hyperplane $\Pi$ is a cone with vertex $\Lambda$. On the other hand, it is defined by $g=0$, hence $g$ is independent of $X_0,\ldots,X_{\lambda}$. So we get that the polynomial $f$ is of the form $$X_{\lambda+1}.h(X_0,\ldots,X_N)+g(X_{\lambda+2},\ldots,X_N).$$ We conclude that there exists a family of hypersurfaces of dimension $${N+d-1\choose N}+{N+d-\lambda-2\choose N-\lambda-2}-1,$$ such that each member $X$ has star points in $\Lambda\setminus (\Lambda\cap\text{Sing}(X))$. If $X$ is a general element of the family, its singular locus is given by $X_{\lambda+1}=\ldots=X_N=h(X_0,\ldots,X_N)=0$, so $\text{Sing}(X)$ is a hypersurface of degree $d-1$ in $\Lambda$.
\end{example}

\section{Extremal case} \label{sec lambda maximal}

In this section, we consider the case where $\lambda$ is maximal, i.e. $\lambda=N-2$ (see Theorem \ref{thm main}).
Assume there do exist two such linear
subspaces $\Lambda_1$ and $\Lambda_2$ of $X$. Let $\Pi_1$ and $\Pi_2$ be the
corresponding hyperplanes in $\mathbb{P}^N$, then we have that $\dim (\Pi_2\cap
\Lambda_1)=N-3$. Since $\Pi_2\cap\Lambda_1\subset \Pi_2\cap X=\Lambda_2$, we get that
$\dim (\Lambda_1\cap \Lambda_2)=N-3$ and
therefore $\dim (\langle \Lambda_1, \Lambda_2\rangle )=N-1$.

Assume $\Lambda_3$ is another such linear subspace of dimension $N-2$ and
assume $\Lambda_3$ does not contain $\Lambda_1\cap \Lambda_2$. Since $\dim (\Lambda_1\cap
\Lambda_3)=N-3$ and $\dim (\Lambda_2 \cap \Lambda_3)=N-3$, it follows that $\Lambda_3 \subset
\langle \Lambda_1,\Lambda_2\rangle$.

\begin{proposition} \label{prop extremal}
Let $X$ be an irreducible reduced hypersurface in $\mathbb{P}^N$ and let $\mathcal{S}$
be the set of $(N-2)$-dimensional linear subspaces $\Lambda$ of $X$ such that
a general element of $\Lambda$ is a star point of $X$.
If $\mathcal{S}$ has at least two elements, then one of the following two possibilities holds.
\begin{enumerate}
\item \label{case L} There exists a linear subspace $L$ of dimension $N-3$
such that all such $\Lambda\in \mathcal{S}$ do contain $L$.
\item \label{case Pi} There exists a linear subspace $H$ of dimension $N-1$
such that all such $\Lambda\in \mathcal{S}$ are contained in $H$.
\end{enumerate}
\end{proposition}

\begin{remark}
Cones over plane curves do give examples of Case \ref{case L} in Proposition \ref{prop extremal}.
If we are in Case \ref{case L} and if $\mathcal{S}$ has at least $d$ elements, the hypersurface $X$
will automatically be a cone over a plane curve with vertex $L$. Indeed, assume there are $d$
such linear spaces $\Lambda_1, \ldots, \Lambda_d\in \mathcal{S}$. They correspond to $d$
different tangent hyperplanes $\Pi_1, \ldots, \Pi_d$. Take $P\in L$
general and $Q\in X$ general. It is enough to prove that
$\langle P, Q\rangle \subset X$. Choose a plane $V$ containing $P$
and $Q$ such that $L_i=\Pi_i\cap V$ are $d$ different lines through
$P$. Let $\gamma =V \cap X$ considered as a divisor on $V$. Since $\Pi_i
\cap X=d\Lambda_i$ as a divisor on $\Pi_i$ and $L_i$ is a line on $\Pi_i$
through $P$, it follows that $D_i=dP$ as a divisor on $L_i$ is a
closed subscheme of $\gamma$. Let $V'$ be the blowing-up of $V$
at $P$, let $E$ be the associated exceptional divisor on $V'$ and
let $L_i'$ (resp. $\gamma '$) be the proper transform of $L_i$
(resp. $\gamma$) on $V'$. In case the multiplicity of $\gamma$ would be smaller than $d$, we have that $\gamma'$ contains $L_i'\cap E$
for $1\leq i\leq d$. Since the lines $L_i$
are $d$ different lines, it follows that $\gamma'$ contains $d$
different points of $E$, which contradicts the assumption. This implies that the
multiplicity of $\gamma$ at $P$ is equal to $d$, hence $\langle P,Q\rangle \subset \gamma\subset X$
since $Q\in \gamma$.

It follows that in this case the set $\mathcal{S}$ contains at most $3d$ linear subspaces, since
a plane curve of degree $d$ has at most $3d$ total inflection points (see e.g. \cite[IV, Ex. 2.3(e)]{Hart}). Equality holds for the cone over the Fermat curve of degree $d$.
\end{remark}

\begin{remark}
In Case \ref{case Pi}, the set $\mathcal{S}$ contains at most $d$ such linear subspaces,
since $H \cap X$ is a divisor of degree $d$ in $H$.
\end{remark}

\begin{example}
Assume $X$ is a hypersurface satisfying Case \ref{case Pi} and $\mathcal{S}$ has $d$ elements $\Lambda_1,\ldots,\Lambda_d\subset H$ with corresponding to tangent spaces $\Pi_1,\ldots,\Pi_d$. We can choose projective coordinates $(X_0:\ldots:X_N)$ on $\mathbb{P}^N$ such that $H$ has equation $X_N=0$ and $\Pi_i$ has equation $l_i(X_0,\ldots,X_N)=0$, hence $\Lambda_i$ is defined by $X_N=l_i=0$. One can see that the equation of $X$ is of the form $$f(X_0,\ldots,X_N)\equiv \prod_{i=1}^d\,l_i(X_0,\ldots,X_N)+\alpha X_N^d=0$$ with $\alpha\in\mathbb{C}$, hence there is a $1$-dimensional family of hypersurfaces having star points in the general points of $\Lambda_1,\ldots,\Lambda_d$.
\end{example}

\section{Surfaces in $\mathbb{P}^3$ with a line of star points} \label{sec surface}

In this section, we consider the special case where $N=3$ and $\lambda=1$, so $X$ is a surface of degree $d\geq 3$ in $\mathbb{P}^3$ and $\Lambda$ is a line on $X$ such that the smooth points of $X$ contained in $\Lambda$ are star points of $X$. We can take projective coordinates $(X_0:X_1:X_2:X_3)$ on $\mathbb{P}^3$ such that $\Lambda$ has equation $X_2=X_3=0$ and $\Pi$ has equation $X_2=0$. From Example \ref{example coordinates} follows that the equation of $X$ can be written as $f\equiv X_2.h(X_0,X_1,X_2,X_3)+g(X_3)=0$. We may take $g(X_3)\equiv X_3^d$. If we write $h(X_0,X_1,X_2,X_3)$ as $$X_3 L_3(X_0,X_1,X_2,X_3)+X_2 L_2(X_0,X_1,X_2)+L(X_0,X_1),$$ the intersection $\text{Sing}(X)\cap \Lambda$ is defined by $X_2=X_3=L(X_0,X_1)=0$. Assume $\text{Sing}(X)\cap \Lambda$ consists of $d-1$ different points. We can choose the coordinates so that $P_0=(1:0:0:0)$ is one of those points, so $L(1,0)=0$, hence $L(X_0,X_1)=X_1L_1(X_0,X_1)$ and $L_1(1,0)\neq 0$. Consider the affine coordinates $(x_1,x_2,x_3)$ on the chart $X_0\neq 0$ as local coordinates around $P_0$ (so $x_i=X_i/X_0$). Using these coordinates, $X$ is defined by $$f\equiv x_2(x_3l_3(x_1,x_2,x_3)+x_2l_2(x_1,x_2)+x_1l_1(x_1))+x_3^d,$$ where $l_i$ is the polynomial corresponding to the form $L_i$. The transformation defined by
\begin{equation*} \left\{ \begin{array}{lll} y_1=x_3l_3(x_1,x_2,x_3)+x_2l_2(x_1,x_2)+x_1l_1(x_1) \\ y_2=x_2 \\ y_3=x_3 \end{array}\right..\end{equation*} is a transformation of local coordinates since $l_1(0)\neq 0$. The equation of $X$ in the local coordinate system $(y_1,y_2,y_3)$ is $y_1y_2+y_3^d$. Using a second transformation of local coordinates defined by
\begin{equation*} \left\{ \begin{array}{lll} z_1=\frac{y_1+y_2}{2} \\ z_2=\frac{i(y_1-y_2)}{2} \\ z_3=y_3 \end{array}\right., \end{equation*}
the equation of $X$ becomes $z_1^2+z_2^2+z_3^d=0$, hence $P_0$ is an $A_{d-1}$-singularity of $X$. We conclude that the $d-1$ points in $\text{Sing}(X)\cap \Lambda$ are $A_{d-1}$-singularities, if they are pairwise different. The following proposition states that the converse also holds.

\begin{proposition} \label{prop d-1 A_d-1}
Let $X\subset \mathbb{P}^3$ be a surface of degree $d\geq 3$. Assume $\Lambda$ is a line on $X$ such that $\Lambda\not\subset\text{Sing}(X)$ and assume $\Lambda$ contains $d-1$ $A_{d-1}$-singularities of $X$. Then a general point of $\Lambda$ is a star point of $X$.
\end{proposition}
\begin{proof}
We can choose projective coordinates on $\mathbb{P}^3$ so that $\Lambda$ is given by $X_2=X_3=0$. The equation of $X$ can be written as $$X_3g(X_0,X_1,X_3)+X_2(X_3L_3(X_0,X_1,X_2,X_3)+X_2L_2(X_0,X_1,X_2)+L(X_0,X_1)).$$ The intersection $\Sing(X)\cap\Lambda$ is defined by $$X_2=X_3=g(X_0,X_1,0)=L(X_0,X_1)=0$$ and contains $d-1$ points, so there exists a complex number $\alpha$ such that $g(X_0,X_1,0)\equiv \alpha L(X_0,X_1)$. The tangent space $\Pi$ at a smooth point on $\Lambda$ has equation $X_2+\alpha X_3=0$. We can change the coordinates on $\mathbb{P}^3$ so that $\alpha=0$, hence $g(X_0,X_1,0)\equiv 0$. Write $X_3 g(X_0,X_1,X_3)=X_3^m G(X_0,X_1,X_3)$ where $m$ is maximal. Note that $m\geq 2$ and we need to prove that $m=d$. Let $\Gamma$ be the curve defined by $X_2=G(X_0,X_1,X_3)=0$, hence $X\cap\Pi=m\Lambda+\Gamma$ as divisors on $\Pi$. Since $\Gamma\cap\Lambda$ consists of at most $d-m<d-1$ points, we can take a point $P\in\text{Sing}(X)\cap\Lambda$ such that $P\not\in \Gamma\cap\Lambda$.    We can also take the coordinates on $\mathbb{P}^3$ such that $P=(1:0:0:0)$. It is easy to see that $P$ is an $A_{m-1}$-singularity, hence $m=d$.
\end{proof}

\begin{remark}
Note that in the above Proposition, it suffices to assume that $\Lambda$ is a line containing $d-1$ $A_{d-1}$-singularities of $X$, since a line containing $d-1$ singular points of a surface $X$ is contained in $X$. Moreover, an $A_{d-1}$-singularity is isolated, so $\Lambda\not\subset\text{Sing}(X)$.
\end{remark}
%Of course, in case $X$ is a surface of degree $d\geq 3$ in $\mathbb{P}^3$ and $\Lambda$ is a line in $\mathbb{P}^3$ such that $\Lambda$ contains (at least) $d-1$ singular points of $X$, then $\Lambda$ is contained in $X$. Hence if those points are $A_{d-1}$-singularities, then the proposition applies. Note that we do not even need to assume that the line $\Lambda$ is not contained in the singular locus of $X$, since $A_{d-1}$-singularities are isolated.

In case of cubic surfaces, the above example was already observed in \cite{Cia}. In that paper, the author expects this is the only case giving rise to cubic surfaces having infinitely many star points. However, this is not true. Indeed, as a trivial example, one can consider a cone $X$ over a cubic curve $\Gamma$. In this case, a smooth inflection point on $\Gamma$ corresponds to a line of star points with exactly one singular point of $X$ and this singularity is not even rational. More generally, assume $\Lambda$ contains only one singular point of a cubic surface $X$. By taking projective coordinates $(X_0:X_1:X_2:X_3)$ on $\mathbb{P}^3$ as before (such that $\Lambda$ has equation $X_2=X_3=0$ and $\Pi$ has equation $X_2=0$), we can write the equation of the surface $X$ as $$f\equiv X_2(X_3 L_3(X_0,X_1,X_2,X_3)+X_2 L_2(X_0,X_1,X_2)+L(X_0,X_1))+X_3^3=0,$$ where $L$ is a square. We may assume that $L(X_0,X_1)=X_1^2$, hence the only singular point of $X$ on $\Lambda$ is $P=(1:0:0:0)$. The intersection of the plane $X_3=0$ and the surface consists of the line $\Lambda$ and the conic with equation $X_3=X_2L_2(X_0,X_1,X_2)+X_1^2=0$. This conic is non-singular if $L_2(1,0,0)\neq 0$. In that case, the surface $X$ has a unique singular point on $\Lambda$ (of type $A_5$ if $L_3(0,0,0,1)\neq 0$ and type $E_6$ if $L_3(0,0,0,1)=0$), hence this case is also different from the example in \cite{Cia}. Of course, both examples are specialisations of the example in \cite{Cia}. In that way, the statement of \cite{Cia} can be adjusted and generalised as follows.

\begin{theorem}
Let $X$ be a surface of degree $d\geq 3$ in $\mathbb{P}^3$ and assume there is an irreducible curve $\Lambda$ on $X$ such that a general point of $\Lambda$ is a star point on $X$. Then $\Lambda$ is a line in $\mathbb{P}^3$ and there exists a $1$-parameter family $(X(t),\Lambda(t))$ with $X(0)=X$, $\Lambda(0)=\Lambda$ and such that for $t\neq 0$, $X(t)$ is a surface of degree $d$ in $\mathbb{P}^3$ and $\Lambda(t)$ is a line on $X(t)$ containing $d-1$ $A_{d-1}$-singularities of $X(t)$. In particular, a general point on $\Lambda(t)$ is a star point on $X(t)$.
\end{theorem}

In a particular case of a surface $X$ with a line $\Lambda$ of star points, we can say something about the types of the singularities of $X$ on $\Lambda$ if there are less than $d-1$ singular points of $X$ on $\Lambda$.

\begin{lemma} \label{lem A_k}
Let $X$ be a surface in $\mathbb{P}^3$ of degree $d\geq 3$ defined by an equation of the form $$f\equiv X_2(X_3X_0^{d-2}+X_1^{\alpha}L(X_0,X_1))+X_3^d=0,$$ with $L(1,0)\neq 0$. Then $P=(1:0:0:0)$ is a singularity of $X$ of type $A_{d\alpha-1}$.
\end{lemma}
\begin{proof}
We are going to work with the affine coordinates $(x_1,x_2,x_3)$ on the chart $X_0\neq 0$ (so $x_i=X_i/X_0$). The equation of $X$ becomes $$f\equiv x_2(x_3+x_1^{\alpha}l(x_1))+x_3^d=0,$$ with $l(x_1)$ the polynomial corresponding to the form $L(X_0,X_1)$. We are interested in the singularity $P$ in the origin. Consider the transformation \begin{equation*} \left\{ \begin{array}{lll} y_1=x_1 \\ y_2=x_2 \\ y_3=x_3+x_1^{\alpha}l(x_1) \end{array}\right.\end{equation*} of local coordinates around $P$. In the local coordinate system $(y_1,y_2,y_3)$, the equation of $X$ is given by
\begin{align*} f &\equiv y_2 y_3+[y_3-y_1^{\alpha}l(y_1)]^d \\
&\equiv y_2 y_3+ \sum_{i=1}^d\,{d\choose i} y_3^i (-y_1^{\alpha}l(y_1))^{d-i} + (-l(y_1))^d y_1^{d\alpha} \\
&\equiv y_3\left[y_2+\sum_{i=1}^d\,{d\choose i}y_3^{i-1}(-y_1^{\alpha}l(y_1))^{d-i}\right] + (-l(y_1))^d y_1^{d\alpha} = 0.
\end{align*}
In the local coordinate system defined by \begin{equation*} \left\{ \begin{array}{lll} z_1=y_1 \\ z_2=y_2+\sum_{i=1}^d\,{d\choose i}y_3^{i-1}(-y_1^{\alpha}l(y_1))^{d-i} \\ z_3=y_3 \end{array}\right.,\end{equation*} the equation of $X$ becomes $z_2 z_3+(-l(z_1))^d z_1^{d\alpha} = 0$. Since $l(0)\neq 0$, there exists a power series $l'(x)$ such that $(l'(x))^{\alpha}=-l(x)$. Finally, in the coordinate system defined by \begin{equation*} \left\{ \begin{array}{lll} w_1=l'(z_1)z_1 \\ w_2=\frac{z_2+z_3}{2} \\ w_3=\frac{i(z_2-z_3)}{2} \end{array}\right.,\end{equation*} the surface $X$ is locally given by $w_1^{d\alpha}+w_2^2+w_3^2=0$, hence $P$ is a singularity of type $A_{d\alpha-1}$.
\end{proof}

\begin{proposition}
Let $X$ be a surface in $\mathbb{P}^3$ of degree $d\geq 3$ defined by an equation of the form $$f\equiv X_2(X_3 L_3(X_0,X_1,X_2,X_3)+X_2 L_2(X_0,X_1,X_2)+L(X_0,X_1))+X_3^d=0,$$ where $L$ is a fixed and $L_2, L_3$ are general. Then each smooth point on the line $\Lambda$ defined by $X_2=X_3=0$ is a star point and each singular point $P=(a:b:0:0)$ on $\Lambda$ is of type $A_{k(P)}$, where $k(P)=d\alpha-1$ and $\alpha$ is the  multiplicity of the root $(a:b)$ of $L=0$. So we have that \begin{equation} \tag{$\star$} \label{eq A_k} \sum_{P\in\Lambda\cap\text{Sing}(X)} \frac{k(P)+1}{d}=d-1. \end{equation}
\end{proposition}
\begin{proof}
Lemma \ref{lem A_k} implies that for $L$ fixed, there exist polynomials $L_2$ and $L_3$ such that the singularities of $X$ on the line $\Lambda$ are all of the type $A_k$ and satisfy formula \eqref{eq A_k}. From \cite{Loo,Lya}, it follows that such singularities are the least worse that can deform into $d-1$ singularities of type $A_{d-1}$, hence the generic statement follows.
\end{proof}

The case of cubic surfaces is also investigated in \cite{Ngu4}. Since in that paper (because of moduli reasons) only semi-stable surfaces are considered, for cubic surfaces with infinitely many star points only the example of \cite{Cia} is obtained. In \cite{Ngu4}, the notion of star points is also introduced for singular points on a cubic surface. We can generalise this notion as follows.

\begin{definition}
Let $X$ be an irreducible hypersurface of degree $d\geq 3$ in $\mathbb{P}^N$. A point $P$ on $X$ is called a {\it star point} on $X$ if there exists a hyperplane $\Pi$ in $\mathbb{P}^N$ such that $\Pi\cap X$ (as a scheme) is a cone with vertex $P$.
\end{definition}

Of course, in the case of surfaces in $\mathbb{P}^3$, a point $P$ is a star point on $X$ if and only if there exists a plane $\Pi$ in $\mathbb{P}^3$ such that the reduced scheme associated to $\Pi\cap X$ is a union of lines through $P$. This corresponds to the definition of a star point in \cite[\S4]{Ngu4}. Note that for smooth points on hypersurfaces, the two definitions are equivalent.

In \cite[\S5]{Ngu4}, the notion of a proper star point is introduced. The meaning of this notion is not completely clear because it is defined with respect to a family, but in the situation of e.g. \cite[Proposition 5.3]{Ngu4} there is no family. Probably, what the author means is the following (we give the general definition).

\begin{definition}
A star point $P$ on $X$ is a {\it proper star point} if there exists a $1$-parameter family $(X(t),P(t))$ such that $X(0)=X$, $P(0)=P$ and $P(t)$ is a smooth star point on $X(t)$ for $t\neq 0$.
\end{definition}

\begin{lemma}
Each star point on a hypersurface of degree $d\geq 3$ is proper.
\end{lemma}
\begin{proof}
Let $P$ be a star point on $X\subset \mathbb{P}^N$. Using coordinates on $\mathbb{P}^N$, we can assume $P=(1:0:\ldots:0)$ and $\Pi$ is the hyperplane in $\mathbb{P}^N$ with equation $X_1=0$ (such that $\Pi\cap X$ is a cone with vertex $P$). From Example \ref{example coordinates} follows that $X$ has equation $X_1 h(X_0,\ldots,X_N)+g(X_2,\ldots,X_N)=0$ with $h$ (resp. $g$) homogeneous of degree $d-1$ (resp. $d$). If we take general homogeneous forms $H$ of degree $d-1$ and $G$ of degree $d$, the surface with equation $X_1 H(X_0,\ldots,X_N)+G(X_2,\ldots,X_N)=0$ is smooth and $P$ is a star point on it. Now consider $X(t)$ to be the surface with equation $X_1(h+tH)+(g+tG)=0$.
\end{proof}

This proof is clearly much shorter than the proof in \cite[\S5]{Ngu4}. In that proof, the cases of $P$ being a singular point of type $A_2$, $P$ a star point on a line connecting two $A_1$-singularities and $P$ a star point on a line connecting two $A_2$-singularities are handled separately. For the first two cases, a blowing-up of $\mathbb{P}^2$ in $6$ points is used to obtain the result of the Lemma. In the third case, a sharper statement is obtained: a star point on a line connecting two $A_2$-singularities is the limit of a star point on a line connecting two $A_1$-singularities. We are going to generalise this result.

\begin{theorem}
Let $X$ be a surface of degree $d\geq 3$ in $\mathbb{P}^3$ and let $\Lambda\not\subset\text{Sing}(X)$ be a line containing $d-1$ singularities of type $A_{d-1}$. Let $P$ be a smooth star point of $X$ on $\Lambda$. Then there is a $1$-parameter family $(X(t),\Lambda(t),P(t))$ with $(X(0),\Lambda(0),P(0))=(X,\Lambda,P)$ such that for $t\neq 0$ the line $\Lambda(t)$ contains $d-1$ $A_{d-2}$-singularities of the surface $X(t)$ in $\mathbb{P}^3$ of degree $d$ and $P(t)$ is a smooth star point of $X(t)$ on $\Lambda(t)$.
\end{theorem}
\begin{proof}
First we make a note on surfaces $X$ of degree $d$ in $\mathbb{P}^3$ containing a line $\Lambda\not\subset\text{Sing}(X)$ such that $\Lambda$ contains $d-1$ singularities of type $A_{d-2}$. In the proof of Proposition \ref{prop d-1 A_d-1}, instead of $m=d$, one obtains $m=d-1$, hence the equation of $X$ can be written as $$X_3^{d-1} g(X_0,X_1,X_3)+X_2 h(X_0,X_1,X_2,X_3)=0$$ and $g(X_0,X_1,0)\not\equiv 0$. The intersection of $X$ and the plane $\Pi$ with equation $X_2=0$ is equal to $(d-1)\Lambda+L$, where the line $L$ has equation $X_2=g(X_0,X_1,X_3)=0$. The set $\text{Sing}(X)\cap \Lambda$ on the line $\Lambda$ is given by $L(X_0,X_1)=0$. In case the zero of $g(X_0,X_1,0)$ on $\Lambda$ is also a zero of $L(X_0,X_1)$, the corresponding point is a $A_{d-1}$-singularity, hence $L$ intersects $\Lambda$ at a smooth point $P$ of $X$. Clearly, the point is a star point of $X$.

Now assume $X$ contains a line $\Lambda\not\subset\text{Sing}(X)$ containing $d-1$ singularities of type $A_{d-1}$. As in the proof of Proposition \ref{prop d-1 A_d-1}, using suited coordinates, the equation of $X$ can be written as $$X_3^d+X_2(X_3L_3+X_2L_2+L(X_0,X_1))=0.$$ The $d-1$ $A_{d-1}$-singularities on $\Lambda$ are the zeroes of $L(X_0,X_1)=0$. Now take a point $P=(\alpha_0:\alpha_1:0:0)$ on $\Lambda$ with $L(\alpha_0,\alpha_1)\neq 0$, so $P$ is smooth on $X$. Let $g(X_0,X_1):=\alpha_1X_0-\alpha_0X_1$ and consider the surface $X(t)$ with equation $$X_3^{d-1}(X_3+t g(X_0,X_1))+X_2(X_3L_3+X_2L_2+L(X_0,X_1))=0.$$ Then $\Lambda(t):=\Lambda$ is a line on $X(t)$ not contained in $\text{Sing}(X(t))$ and for $t\neq 0$ the surface $X(t)$ has exactly $d-1$ $A_{d-2}$-singularities on $\Lambda$. From the previous discussion, the point $P(t):=P$ is smooth on $X(t)$.
\end{proof}

Now consider a cubic surface $X$ in $\mathbb{P}^3$ and denote by $\mathcal{S}$ the closure of the union of all $1$-dimensional families of star points on $X$. From Proposition \ref{prop extremal} follows that there are only two possibilities: either $X$ is a cone on a cubic plane curve and then $\mathcal{S}$ is the union of lines through the vertex of $X$, either there is a plane $H$ and $\mathcal{S}$ is a union of lines in $H$. Consider the second case. We know each pair of $A_2$-singularities in $X$ give rise to a line inside $\mathcal{S}$, so all $A_2$-singularities are contained in one plane $H$. In case $X$ has at least three $A_2$-singularities, then $\mathcal{S}$ need to be the union of three different lines in $H$. If the intersection $P$ of two such lines would be a smooth point of $X$, since $X\cap H$ is singular at $P$, we would obtain $H=T_P(X)$. But $T_P(X)\cap X$ should be a cone (since $P$ is a star point), hence $P$ is not smooth on $X$. Therefore $\mathcal{S}$ are non-concurrent lines and the intersection points are $A_2$-singularities. This shows $X$ has at most three $A_2$-singularities. This is in accordance with known results, see e.g. \cite{Ngu1}.

\section*{Acknowledgements}

The authors like to thank Ciro Ciliberto for pointing us out reference \cite{Cia} and Alex Degtyarev for leading us to references \cite{Loo, Lya}. Both authors are partially supported by the project G.0318.06 of the Fund of Scientific Research - Flanders (FWO) and the first author is a Postdoctoral Fellow of FWO.

\end{document}